\documentclass[11pt,leqno]{article}
\usepackage{amsmath,amssymb,mathrsfs,xy,amsthm,bm,color,url}
\usepackage{stmaryrd}
\input{xy}
\xyoption{all}

\makeatletter
\newcommand{\nsection}{\@startsection{section}{1}{\z@}%
     {-3.5ex plus-1ex minus-.2ex}%
     {2.3ex plus.2ex}%
     {\reset@font\center\large\sc}}
\renewenvironment{thebibliography}[1]
{\nsection*{\refname\@mkboth{\refname}{\refname}}%
   \list{\@biblabel{\@arabic\c@enumiv}}%
        {\settowidth
	\labelwidth{\@biblabel{#1}}%
         \leftmargin
	 \labelwidth
         \advance
	 \leftmargin
	 \labelsep
         \@openbib@code
         \usecounter{enumiv}%
         \let\p@enumiv\@empty
	 \parskip=0pt
	 \itemsep=1pt
	 \parsep=1pt
	 \itemindent=\z@
         \renewcommand\theenumiv{\@arabic\c@enumiv}}%
   \sloppy
   \clubpenalty4000
   \@clubpenalty\clubpenalty
   \widowpenalty4000%
   \footnotesize
   \sfcode`\.\@m}
  {\def\@noitemerr
    {\@latex@warning{Empty `thebibliography' environment}}%
   \endlist}

 \renewcommand{\section}%
  {\@startsection{section}%
  {1}%
  {\z@}%
  {-3.5ex plus -1ex minus -.2ex}%
  {1ex}%
  {\reset@font\large\bfseries}
  }%
 \def\@seccntformat#1{\csname the#1\endcsname.\hspace{2ex}}

\makeatother

\newtheoremstyle{thm}
 {3pt}
 {3pt}
 {\itshape}
 {}
 {\bf}
 {. ---}
 {0.5em}
 {}
\newtheoremstyle{dfn}
 {3pt}
 {3pt}
 {}
 {}
 {\bf}
 {. {---}}
 {0.5em}
 {}

\swapnumbers
\theoremstyle{thm}
\newtheorem{thm}[subsection]{Theorem}
\newtheorem{lem}[subsection]{Lemma}
\newtheorem{cor}[subsection]{Corollary}

\newtheorem*{prop*}{Proposition}

\newtheorem*{conj*}{Conjecture}
\newtheorem*{thm*}{Theorem}
\newtheorem*{lem*}{Lemma}
\newtheorem*{cor*}{Corollary}

\theoremstyle{dfn}

\newtheorem{rem}[subsection]{Remark}
\newtheorem*{rem*}{Remark}
\newtheorem{cl}[subsection]{Claim}

\newtheorem{rem1}[subsection]{Remark\footnote{
This remark was pointed out to the author by A. Shiho.}}

\newcommand{\mr}[1]{\mathrm{#1}}
\newcommand{\ms}[1]{\mathscr{#1}}
\newcommand{\mc}[1]{\mathcal{#1}}
\newcommand{\mb}[1]{\mathbb{#1}}

\newcommand{\DdagQ}[1]{\mathscr{D}^\dag_{{#1},\mathbb{Q}}}

\newcommand{\indlim}{\mathop{\underrightarrow{\mathrm{lim}}}}

\newcommand{\loc}{\mc{K}_{\mr{loc}}}

\begin{document}
\title{Langlands program for $p$-adic coefficients and the petites
camarades conjecture}
\author{Tomoyuki Abe}
\date{}
\maketitle

\begin{abstract}
 In this paper, we prove that, if Deligne's ``petites camarades
 conjecture'' holds, then a Langlands type correspondence holds also for
 $p$-adic coefficients on a smooth curve over a finite field. We also
 prove that any overconvergent $F$-isocrystal of rank less than or equal
 to $2$ on a smooth variety is $\iota$-mixed.
\end{abstract}

\section{Introduction}
Let $p$ be a prime number, and $k$ be a finite field with $q=p^s$
elements. Let $X$ be a smooth proper geometrically connected curve over
$\mr{Spec}(k)$ with function field $\mc{K}$. First, we put:
\begin{center}
 \begin{tabular}{cp{14cm}}
  $\mc{A}_r$:& The set of isomorphism classes of irreducible cuspidal
  automorphic representations $\pi$ of $\mr{GL}_r(\mb{A}_{\mc{K}})$,
  where $\mb{A}_{\mc{K}}$ denotes the ring of adeles of $\mc{K}$, such
  that the order of the central character of $\pi$ is finite.\\
 \end{tabular}
\end{center}
Let $U$ be an open dense subscheme of $X$. Let $K$ be a finite
extension of $K_0:=W(k)\otimes\mb{Q}$. We fix an algebraic closure
$\overline{\mb{Q}}_p$ of $K_0$. For an overconvergent
$F$-isocrystal $E$ of rank $r$ on a dense open subset $U$ of $X/K$ (see
\S \ref{overconv} for the definition), we say that it is {\em of
finite determinant} if there exists a positive integer $n$ such that the
$n$-th tensor power of $\det(E):=\bigwedge^rE$ is the trivial
overconvergent $F$-isocrystal. We say that $E$ is {\em absolutely
irreducible} if for any finite extension $L$ of $K$ in
$\overline{\mb{Q}}_p$, the isocrystal $E\otimes_KL$ on $U/L$ is
irreducible. We put:
\begin{center}
 \begin{tabular}{cp{13cm}}
  $\mc{I}_r(U/K)$:& The set of isomorphism classes of absolutely
  irreducible overconvergent $F$-isocrystals of rank $r$ on $U/K$ which
  is of finite determinant.\\
 \end{tabular}
\end{center}
For any finite extension $L$ of $K$, the scalar extension induces
a map $\mc{I}_r(U/K)\rightarrow\mc{I}_r(U/L)$, and we put
$\mc{I}(U):=\indlim_{K/K_0}\mc{I}(U/K)$ where $K$ runs over any finite
extension of $K_0$ in $\overline{\mb{Q}}_p$. When we have two open
subschemes $V\subset U$ of $X$, the restriction induces a map
$\mc{I}_r(U)\rightarrow\mc{I}_r(V)$ by Lemma \ref{irreducible}
below. Using this map, we define a set by
\begin{equation*}
 \mc{I}_r:=\indlim_{U\subset X}\mc{I}_r(U).
\end{equation*}

In this paper, we fix two things: an isomorphism
$\iota\colon\overline{\mb{Q}}_p\cong\mb{C}$, and a root $\pi$ of the
equation $x^{p-1}+p=0$ or equivalently a non-trivial additive character
$\mb{F}_p\rightarrow\overline{\mb{Q}}_p^\times$ (cf.\
\cite[2.4.1]{AM}). We conjecture the following:

\begin{conj*}[Langlands program for $p$-adic coefficients]
 There are two maps $\pi_{\bullet}$ and $E_{\bullet}$ as follows:
 \begin{enumerate}
  \item  There exists a {\em unique} map
	 $\pi_{\bullet}\colon\mc{I}_r\rightarrow\mc{A}_r$ such that for
	 $E\in\mc{I}_r$, the sets of unramified places of $E$ and
	 $\pi_E$ coincide, which we denote by $U$, and for any $x\in
	 |U|$, the eigenvalues of the Frobenius of $E$ at $x$ and the
	 Hecke eigenvalues of $\pi_E$ at $x$ coincide.

  \item  There exists a {\em unique} map
	  $E_{\bullet}\colon\mc{A}_r\rightarrow\mc{I}_r$ such that for
	 $\pi\in\mc{A}_r$, the sets of unramified places of $\pi$ and
	 $E_\pi$ coincide, which we denote by $U$, and for any
	 $x\in |U|$, the eigenvalues of the Frobenius of $E_\pi$ at $x$
	 and the Hecke eigenvalues of $\pi$ at $x$ coincide.
 \end{enumerate}
 Moreover, these maps induce one-by-one correspondence, namely
 $E_{\bullet}\circ\pi_{\bullet}=\mr{id}_{\mc{I}}$, and
 $\pi_{\bullet}\circ E_{\bullet}=\mr{id}_{\mc{A}}$.
\end{conj*}

For short, we call this Conjecture (L). In the celebrated paper
\cite[1.2.10 (vi)]{WeilII}, Deligne conjectured existence of ``{\it
petites camarades}'' of smooth $\ell$-adic sheaves. This conjecture was
stated in a vague way, which was later formulated in a clear form by
Crew \cite[4.13]{CrMono} as follows:

\begin{conj*}[petites camarades conjecture for curves]
 Let $U$ be a smooth geometrically connected curve over $k$, and
 $\ms{F}$ be an irreducible smooth $\overline{\mb{Q}}_\ell$-sheaf on $U$
 whose determinant is defined by a
 finite order character of $\pi_1(U)$. Then there exists a number field
 $E$ such that, for any $x\in|U|$, the characteristic polynomial of the
 action of the geometric Frobenius on $\ms{F}_{\overline{x}}$ has
 coefficients in $E$, and for any place $\ms{P}$ of $E$ dividing $p$,
 there is an overconvergent $F$-isocrystal on $U/E_{\ms{P}}$ such that
 the characteristic polynomials of the Frobenius action at $x$ for
 $\ms{F}$ and $E$ coincide for any $x\in|U|$.
\end{conj*}
For short, we call this Conjecture (D). By using the proven Langlands
correspondence in the $\ell$-adic cases \cite[VI.9]{La}, if Conjecture
(L) holds, Conjecture (D) holds and moreover, we can take the
overconvergent $F$-isocrystal in Conjecture (D) to be
irreducible. The main theorem of this paper is the following:

\begin{thm*}
 Conjecture {\normalfont (L)} and {\normalfont (D)} are equivalent.
\end{thm*}

The proof of this theorem have been made possible thanks to the product
formula for $p$-adic epsilon factors proven in \cite{AM}, and we use the
standard argument to deduce the theorem from the product
formula (cf.\ \cite[9.7]{Delconst} for $\mr{GL}_2$ case, and
\cite[3.2.2.3]{Lau} for $\mr{GL}_n$ case).

In the last part of this paper, we prove that any overconvergent
$F$-isocrystal of rank less than or equal to $2$ on a smooth variety is
$\iota$-mixed, which can be seen as a part of the results we can derive
from Conjecture (L).

\section{Overconvergent $F$-isocrystals}
\label{overconv}
For the Langlands correspondence, we need to consider overconvergent
$F$-isocrystals on $U/K$ where $U$ is an open dense subscheme of $X$ and
$K$ is a finite extension of $K_0$. When $K$ is not totally ramified
over $K_0$, this concept is not defined {\it a priori}, so we briefly
review the concept in this section although it might be trivial for
experts.

Let $Y$ be a scheme of finite type over $k$.
We consider the $s$-th Frobenius automorphism $F$ on $k$, which is the
identity since the number of elements of $k$ is $p^s$. We denote by
$F\mbox{-}\mr{Isoc}^\dag(Y/K_0)$ the category of overconvergent
$F$-isocrystals on $Y/K_0$ (cf.\ \cite{Be}).
We define the category of overconvergent
$F$-isocrystals on $Y/K$ denoted by $F\mbox{-}\mr{Isoc}^\dag(Y/K)$
to be $(F\mbox{-}\mr{Isoc}^\dag(Y/K_0))_{K}$ using the notation of
\cite[7.3]{AM} ({\it i.e.}\ the category of couples $(E_0,\lambda)$ such
that $E_0\in\mr{Ob}\bigl(F\mbox{-}\mr{Isoc}^\dag(Y/K_0)\bigr)$ and
$\lambda\colon K\rightarrow\mr{End}(E_0)$, which is called the
{\em $K$-structure of $E$}, is a homomorphism of
$K_0$-algebras). If $K$ is a totally ramified extension of $K_0$, the
category $F\mbox{-}\mr{Isoc}^\dag(Y/K)$ is nothing but that defined in
\cite{Be}. We see easily that the
category $F\mbox{-}\mr{Isoc}^\dag(\mr{Spec}(k)/K)$ is equivalent to the
category of finite dimensional $K$-vector spaces $V$ endowed with an
automorphism of $K$-vector spaces $V\xrightarrow{\sim}V$. Let $x$ be a
closed point of $Y$, and $i_x\colon \{x\}\hookrightarrow Y$ be the
closed immersion. Then using the equivalence, $i_x^*E$ is a finite
dimensional $K$-vector space endowed with an automorphism, and we are
able to consider the characteristic polynomial of the Frobenius action
of $E$ at $x$.

For an object $E=(E_0,\lambda)$ of
$F\mbox{-}\mr{Isoc}^\dag(Y/K)$, we note that $\mr{rk}(E_0)$ is divisible
by $[K:K_0]$, and we define the rank of $E$ to be
$\mr{rk}(E):=[K:K_0]^{-1}\cdot\mr{rk}(E_0)$. As written in
\cite[7.3.5]{AM}, we automatically have rigid cohomology
$H^r_{\mr{rig}}(Y,E)$ which is a {\em $K$-vector space with Frobenius
structure} and so on.

Now, let $x\in X\setminus U$. Let $K_x$ be the unramified extension of
$K_0$ corresponding to the extension $k(x)/k$. An overconvergent
$F$-isocrystal $E_0$ on $U/K_0$ induces a differential module on the
Robba ring $\mc{R}_{K_x}$ over $K_x$ with Frobenius structure. We denote
this by $E_0|_{\eta_x}$. Given an overconvergent $F$-isocrystal
$E=(E_0,\lambda)$ on $U/K$, we put $E|_{\eta_x}$ to be the differential
module with Frobenius structure $E_0|_{\eta_x}$ endowed with
$K$-structure. We see that $\mr{irr}(E_0|_{\eta_x})$ is divisible by
$[K:K_0]$, and we define the irregularity of $E$ at $x$ to be
$\mr{irr}(E|_{\eta_x}):=[K:K_0]^{-1}\cdot\mr{irr}(E_0|_{\eta_x})$.

Now, we prove the following irreducibility result\footnote{
The proof written here was suggested by N. Tsuzuki. One can refer to
\cite{AC} for more general results on the irreducibility.
}:

\begin{lem}
 \label{irreducible}
 Let $W$ be a smooth curve over $k$, and $V\subset U\subset W$ be dense
 open subschemes of $W$, and let $j\colon V\hookrightarrow U$ be the
 open immersion. Then the functor $j^\dag\colon
 F\mbox{-}\mr{Isoc}^\dag(U,W/K)\rightarrow
 F\mbox{-}\mr{Isoc}^\dag(V,W/K)$ preserves irreducibility.
\end{lem}
\begin{proof}
 In the following argument, all the modules are implicitly equipped with
 $K$-structure. Let $E$ be an irreducible object in
 $F\mbox{-}\mr{Isoc}^\dag(U,W/K)$. Assume there exists a non-zero
 overconvergent $F$-isocrystal $E'_V$ and a surjection $j^\dag
 E\twoheadrightarrow E'_V$. Then we can find a finite covering $f\colon
 W'\rightarrow W$ such that the following diagram is commutative
 \begin{equation*}
  \xymatrix{
   V'\ar[r]^{j'}\ar[d]_{f_V}\ar@{}[rd]|\square&U'\ar[d]^f\ar[r]
   \ar@{}[dr]|\square&W'\ar[d]^{f_W}
   \\V\ar[r]_j&U\ar[r]&W
   }
 \end{equation*}
 where $f_V$ is finite \'{e}tale and $f_V^*E'_V$ is log-extendable to
 $U'$ by \cite{Kedcss}. Since $f_V^*E\twoheadrightarrow f_V^*E'_V$,
 $f_V^*E'_V$ extends to an overconvergent $F$-isocrystal $E'_{U'}$ on
 $U'$ which is a quotient of $f^*E$. Let $\ms{W}$ and $\ms{W}'$ be
 smooth formal liftings of $W$ and $W'$ respectively over
 $\mr{Spf}(W(k))$. Let $Z:=W\setminus U$ and
 $Z':=W'\setminus U'$. We have a coherent $\DdagQ{\ms{W}}(^\dag
 Z)$-module $\mr{sp}_*(E)$ and a coherent $\DdagQ{\ms{W'}}(^\dag
 Z')$-module $\mr{sp}_*(E'_{U'})$. Since $f_W$ is proper, we have a
 functor $f_+\colon D_{\mr{coh}}^b(\DdagQ{\ms{W'}}(^\dag
 Z'))\rightarrow D_{\mr{coh}}^b(\DdagQ{\ms{W}}(^\dag Z))$ (see
 \cite[3.4.3]{Be2} for the construction of the functor in the case where
 $f_W$ is not liftable). Then we get a homomorphism
 \begin{equation*}
  \phi\colon\ms{H}^0(f_+f^!\mr{sp}_*(E))\rightarrow
   \ms{H}^0(f_+\mr{sp}_*(E'_{U'})).
 \end{equation*}
 By using the trace homomorphism (see \cite[4.15]{Abe}), $\mr{sp}_*(E)$
 can be seen as a submodule of $\ms{H}^0(f_+f^!\mr{sp}_*(E))$. We define
 $E':=\mr{sp}^*(\phi(\mr{sp}_*(E)))$. Since $j^\dag E'\cong E'_V$, $E'$
 is not equal to $E$ if $E'_V$ is not equal to $j^\dag E$. Since $E$ is
 assumed to be irreducible, this shows that $E'_V=j^\dag E$, and
 conclude the proof.
\end{proof}

\begin{rem1}
 Lemma \ref{irreducible} does not hold if we replace $\mr{Isoc}^\dag$
 by $\mr{Isoc}$. For example, let $f\colon
 Y\rightarrow A:=\mb{P}^1\setminus\{0,1,\infty\}$ be the Legendre family
 defined by $y^2=x(x-1)(x-\lambda)$. Consider the $F$-isocrystal
 $E:=R^1f_{\mr{rig}*}(Y/A)$ of rank $2$. Let $S\subset A$ be the
 finite set of closed points $s$ such that $f^{-1}(s)$ is a
 supersingular elliptic curve, and we put $U:=A\setminus S$. Then the
 convergent $F$-isocrystal $E|_U$ is not irreducible as written in
 \cite[4.15]{Crrep}. However, $E$ is irreducible since all the
 subquotients of $E|_U$ are not overconvergent along $S$ as written in
 {\it ibid.}. The same example is showing that the functor
 $F\mbox{-}\mr{Isoc}^\dag(U,A/K)\rightarrow F\mbox{-}\mr{Isoc}(U/K)$
 does not preserve irreducibility.
\end{rem1}

\section{\v{C}ebotarev density}
\label{Cebotarev}
A technical point in the proof of the theorem is to show the
uniqueness of the maps. In the $\ell$-adic case, this was a consequence
of the \v{C}ebotarev density theorem. Since in general we are not able
to describe overconvergent $F$-isocrystals as representations of
$\mr{Gal}(\overline{\mc{K}}/\mc{K})$, we need a different method to show
this. We will show the following proposition\footnote{After a large part
of this paper had been written, the author was pointed out by A.\
P\'{a}l that the following proposition had been generalized by him and
U. Hartl. However, since we do not know the precise statement of their
theorem, we decided to put the proposition.}:
\begin{prop*}
 \label{p-adicCebo}
 Assume $E$ and $E'$ are {\em $\iota$-mixed} (cf.\
 {\normalfont\cite[10.4]{Cr}} for
 the definition) overconvergent $F$-isocrystals on an open dense
 subscheme $U/K$ of $X$ whose characteristic polynomials of the action
 of the Frobenius at any closed point of $U$ coincide. Then the
 semi-simplifications of $E$ and $E'$ coincide.
\end{prop*}

Before proving the proposition, we recall some notation in
\cite[10.4, 10.10]{Cr}. Let $M:=E^{\mr{ss}}\oplus E'^{\mr{ss}}$ where
${}^{\mr{ss}}$ denotes the semi-simplification, and take a closed point
$x_0$ in $U$. Associated to this overconvergent $F$-isocrystal, we have
a short exact sequence:
\begin{equation*}
 0\rightarrow\mr{DGal}(M,x_0)\rightarrow W^{M}_{x_0}\rightarrow
  W(\overline{k}/k)\rightarrow0.
\end{equation*}
Here, $\mr{DGal}(M,x_0)$ denotes the differential Galois group
of $M$ (see \cite[2.1]{CrMono}), $W^M_{x_0}$ denotes the
Weil group of $M$ (see [{\it ibid.}, 5.1]), and
$W(\overline{k}/k)\cong\mb{Z}$. We denote the extension of
scalars by $K\rightarrow\mb{C}$ using the isomorphism $\iota$ of
$\mr{DGal}(M,x_0)$ and $W^M_{x_0}$ by $G^0_{\mb{C}}$ and $G_{\mb{C}}$
respectively. Since $M$ is assumed to be $\iota$-mixed, there is a
subgroup $G_{\mb{R}}\subset G_{\mb{C}}$ which projects onto
$W(\overline{k}/k)$ and $G^0_{\mb{C}}\cap G_{\mb{R}}$ is a maximal
compact subgroup of $G^0_{\mb{C}}$. For a group $G$, we denote by
$G^{\natural}$ the set of conjugacy classes. We take an element $z$ of
positive degree in the center of $G_{\mb{R}}$, which exists since $M$ is
semi-simple. For more details, see \cite[10.10]{Cr}.

\begin{proof}
 Let $\rho^{(\prime)}\colon W^M_{x_0}\rightarrow\mr{GL}
 (V_{\rho^{(\prime)}})$ be a representation of $W^M_{x_0}$
 corresponding to $(E^{(\prime)})^{\mr{ss}}$. These define
 continuous complex representations of $G_{\mb{R}}$ denoted by
 $\rho_{\mb{R}}$ and $\rho'_{\mb{R}}$. Since the category of linear
 $\overline{\mb{Q}}_p$-representations of $W^M_{x_0}$
 and that of continuous complex representations of $G_{\mb{R}}$ are
 equivalent by exactly the same argument as \cite[2.2.8]{WeilII}, it
 suffices to show that $\rho_{\mb{R}}\cong\rho'_{\mb{R}}$.

 We endow the quotient topology with $G_{\mb{R}}^{\natural}$ by the
 surjection $\alpha\colon G_{\mb{R}}\twoheadrightarrow
 G_{\mb{R}}^{\natural}$. We denote by $G^{\natural}_i$ the subset of
 $G^{\natural}_{\mb{R}}$ consisting of the elements of degree $i$. Note
 that $G^{\natural}_i$ is both open and closed. Let $S$ be the subset of
 elements $g$ of $G^{\natural}_{\mb{R}}$ such
 that the characteristic polynomials $\det(1-g\cdot t;V_{\rho})$ and
 $\det(1-g\cdot t;V_{\rho'})$ coincide. By the Brauer-Nesbitt theorem,
 it suffices to show that
 $S=G^{\natural}_{\mb{R}}$. Since if the characteristic polynomials for
 $g$ coincide, that of $g^{-1}$ coincide as
 well, we have $S=S^{-1}$. Since the map
 $\widetilde{\rho}^{(\prime)}\colon G\rightarrow\mb{C}[t]$ sending $g$ to
 $\det(1-g\cdot t;V_{\rho^{(\prime)}})$ is continuous,
 $S=(\widetilde{\rho}-\widetilde{\rho}')^{-1}(0)$ is closed in
 $G^{\natural}_{\mb{R}}$. Let $A$ be the complement of $S$ in
 $G_{\mb{R}}$, which is open, and {\em assume that this is not
 empty}. Let $A_i:=A\cap G^{\natural}_i$, which is an open subset of
 $G^{\natural}_{\mb{R}}$. Since $z$ is in the center of $G_{\mb{R}}$, the
 characteristic polynomials of $g$ for an element $g\in G_{\mb{R}}$
 coincide if and only if they coincide for $z^{n}\cdot g$ for some integer
 $n$. This is showing that the isomorphism $z\colon
 G^{\natural}_{n}\xrightarrow{\sim} G^{\natural}_{n+d}$ where $d>0$
 denotes the degree of $z$ induces the bijection
 $A_{n}\xrightarrow{\sim}A_{n+d}$. Since we are assuming $A$ to be
 non-empty, there exists a {\em positive} integer $n_0$ such that
 $A_{n_0}$ is also non-empty.

 Now, let $dg$ be the Haar measure on $G_{\mb{R}}$, and $\mu_0$ be the
 product of $dg$ by the characteristic function of the elements of positive
 degree. We denote by $\mu^{\natural}_0$ the direct image of $\mu_0$ on
 $G^{\natural}_{\mb{R}}$. Since $A_{n_0}$ is open, we get
 \begin{equation*}
 \mu^{\natural}_0(A_{n_0})=dg(\alpha^{-1}(A_{n_0}))>0.
 \end{equation*}
 Thus the equidistribution theorem \cite[10.11]{Cr}\footnote{
 Although missing in {\it ibid.}, in this theorem, we think that we
 need to assume, moreover, $M$ to be semi-simple.}
 implies that there
 exists a positive integer $n$ such that $\mu^{\natural}(z^n\cdot
 A_{n_0})>0$. This is showing that $z^{-n}\cdot \mr{Frob}_x^{n'}\in
 A_{n_0}$ where $-nd+\deg(x)\,n'=n_0$. As a consequence, we have
 $\mr{Frob}_x^{n'}\in A$, which contradicts with the assumption.
\end{proof}

\section{$L$-factors and $\varepsilon$-factors}
Before proving the main theorem, let us review the theory of local
$L$-factors and $\varepsilon$-factors. Let $\psi$ be a non-trivial
additive character of $\mb{A}_{\mc{K}}/\mc{K}$, which is equivalent to
choosing a meromorphic differential form $\omega$ on $X$. For an
overconvergent $F$-isocrystal $E$ on a dense open subset $U/K$ of $X$,
we have the global $\varepsilon$-factor\footnote{
The definition is slightly different from \cite[7.2]{AM}.}
defined by
\begin{equation*}
 \varepsilon(E,t):=\prod_{r\in\mb{Z}}\det(-F\cdot t;
  H^{r}_{\mr{rig},c}(U,E))^{(-1)^{r+1}}
\end{equation*}
where $F$ denotes the Frobenius automorphism of the rigid cohomology
with compact support. For each closed point $x$ of $X$, we have local
$\varepsilon$-factor defined as follows. We define the (Artin) conductor
of $E$ at $x$ by
\begin{equation*}
 a_x(E_x):=
  \begin{cases}
   \mr{rk}(E)+\mr{irr}(E|_{\eta_x})&
   \mbox{if $x$ is ramified,}\\
   0&\mbox{if $x$ is unramified.}
  \end{cases}
\end{equation*}
Then we define the local $\varepsilon$-factor at $x$ to be
\begin{align*}
 \varepsilon_x(E_x,t,\psi_x)&:=
  \begin{cases}
   \varepsilon^{\mr{rig}}_0(E|_{\eta_x},\omega_x)\cdot c_x(t)
   &\mbox{if $x$ is ramified,}\\
   \bigl(q^{\deg(x)\,v_x(\omega)\,\mr{rk}(E)}\cdot
   \det_E(x)^{v_x(\omega)}\bigr)\cdot c_x(t)
   &\mbox{if $x$ is unramified,}
  \end{cases}\\
 c_x(t)&:=q^{-\deg(x)\,v_x(\omega)\,\mr{rk}(E)/2}\cdot
 t^{\deg(x)(\mr{rk}(E)\,v_x(\omega)+a_x(E))},
\end{align*}
where $\varepsilon^{\mr{rig}}_0(\cdot)$ denotes the $\varepsilon$-factor
defined by Marmora (cf.\ we followed the notation of \cite[7.1.3]{AM}),
$\det_M(x)$ denotes the determinant of the Frobenius action on $E$ at
$x$ (see [{\it ibid.}, 7.2.6]), $v_x(\omega)$ denotes the order of
$\omega$ at $x$. The main theorem of \cite{AM} (Theorem 7.2.6 of {\it
ibid.}) is stating that the following equality holds:
\begin{equation}
 \label{prodform}\tag{PF}
 \varepsilon(E,t)=\prod_{x\in|X|}\varepsilon_x(E_x,t,\psi_x).
\end{equation}

\begin{rem*}
 In \cite{AM}, only the case where the residue field of $K$ is equal to
 $k$ is dealt with. However, we are able to prove the product formula in
 the same way without assuming this,
 or more precisely [{\it ibid.}, (7.4.3.1)]
 holds. The detail is left to the reader. We note that the
 $\varepsilon$-factor in {\it ibid.}\ is a function. However, applying
 the product formula to the canonical extension, we see that the
 function is constant, and we define the $\varepsilon$-factor
 $\varepsilon^{\mr{rig}}_0(E|_{\eta_x},\omega_x)$ in the above to be
 this constant.
\end{rem*}
Let $\loc$ be a local field, and $\pi$ and $\pi'$ be smooth admissible
representations of $\mr{GL}_r(\loc)$ which are irreducible or of
Whittaker type. Fix an additive character $\psi_{\mr{loc}}$ of $\loc$.
Then in \cite[Thm 2.7 and subsection 9.4]{JPS}, the $L$-factor and
$\varepsilon$-factor for the pair $(\pi,\pi')$ are defined, and denoted
by $L(\pi\times\pi',t)$ and
$\varepsilon(\pi\times\pi',t,\psi_{\mr{loc}})$ respectively. For
automorphic representations $\pi=\bigotimes_{x\in|X|}\pi_x$ and
$\pi'=\bigotimes_{x\in|X|}\pi'_x$ such that $\pi_x$ and $\pi'_x$ are
irreducible or of Whittaker type for any $x$,
we denote by $L_x(\pi_x\times\pi'_x,t)$ (resp.\
$\varepsilon_x(\pi_x\times\pi'_x,t,\psi_x)$) the local $L$-factor
(resp.\ $\varepsilon$-factor) of the couple $(\pi_x,\pi'_x)$ (resp.\ and
the additive character $\psi_x$ induced by $\psi$) of representations of
$\mr{GL}_r(\mc{K}_x)$ where $\mc{K}_x$ denotes the local field at the
place $x$ of $\mc{K}$.

\section{Proof of the theorem}
We prove the main theorem in this section. First, we remark the
following.

\begin{rem}
 \label{remRPconj}
 If the map $\pi_\bullet$ is constructed for some $r$, the isocrystals
 belonging to $\mc{I}_r$ is $\iota$-pure of weight $0$. This can be seen
 from the generalized Ramanujan-Petersson conjecture proven by Lafforgue
 \cite[VI.10 (i)]{La}.
\end{rem}

Let us start to prove the main theorem. First let us see the uniqueness
of the maps. The map $\pi_{\bullet}$ is uniquely determined if it exists
by the strong multiplicity one theorem \cite{Pia}. For the uniqueness of
$E_{\bullet}$, assume we had two maps $E_{\bullet}$ and
$E'_{\bullet}$. For $\pi\in\mc{A}_r$, $E_{\pi}$ and $E'_\pi$ are
$\iota$-pure of weight $0$ by Remark \ref{remRPconj}, and we are in the
situation to apply the proposition of \S\ref{Cebotarev}, which implies
that $E_\pi=E'_\pi$. Thus the uniqueness of $E_{\bullet}$ follows.

The rest of the argument is nothing but the ``{\it principe de
r\'{e}currence}'' using the product formula (\ref{prodform}). Since all
we need to do is to copy the proof of \cite[VI.11]{La}, we only sketch
the proof. We use the induction on $r$. We claim the following.

\begin{cl}
 \label{indDeligh}
 Assume that the correspondence is established for any $r$ which is
 strictly less than $r_0$, and the local $L$-factor and
 $\varepsilon$-factor coincide {\em for any place of $X$} via this
 correspondence. Then we have the map
 $\pi_{\bullet}\colon\mc{I}_r\rightarrow\mc{A}_r$ in the sense of
 Langlands for $r=r_0$.
\end{cl}

\begin{proof}[Sketch of the proof]
 Take an element $E$ of $\mc{I}_{r_0}$, and denote by $S$ the set of
 points of $X$ at which $E$ is
 ramified. For a point $x\not\in S$, we put $\pi_x$ to be the unramified
 smooth admissible irreducible representation of
 $\mr{GL}_{r_0}(\mc{K}_x)$ whose set of Hecke eigenvalues is that of
 Frobenius eigenvalues of $E$ at $x$. For $x\in S$, we put $\pi_x$ to be
 an irreducible representation of $\mr{GL}_{r_0}(\mc{K}_x)$ of Whittaker
 type whose center corresponds to $\det(E_x)$ via the reciprocity
 map. We put $\pi:=\bigotimes_{x\in|X|}\pi_x$,
 which is a smooth admissible irreducible representation of
 $\mr{GL}_{r_0}(\mb{A}_{\mc{K}})$. First, we need to show that there is
 an irreducible automorphic representation $\pi_E$ such that the local
 factors of $\pi_E$ at any closed point $x\not\in S$ is equal to
 $\pi_x$. For this, we
 use the converse theorem \cite[B.13]{La} of Piatetski-Shapiro. Since
 the argument is exactly the same as \cite{La} using the product formula
 (\ref{prodform}), we omit the detail. We only note here that on a way
 we check the hypothesis of the converse theorem, we need to show that
 for any closed point $x$ of $X$, the local factors coincide for certain
 pairs. When $\pi$ is {\em unramified} at $x$, we use the hypothesis on
 the coincidence of local factors in the statement of the claim to see
 this coincidence at $x$. As a result, we have an automorphic
 representation $\pi_E$ with desired property, which is moreover
 cuspidal if $S=\emptyset$.

 Finally, to show the claim, it remains to show that the
 automorphic representation is in fact cuspidal when
 $S\neq\emptyset$. Assume that $\pi_{E}$ were not cuspidal. Then a
 result of Langlands is saying that there exists a non-trivial partition
 $r_0=r_1+\dots+r_k$, and automorphic cuspidal representations
 $\pi^1,\dots,\pi^{k}$ of $\mr{GL}_{r_1},\dots\mr{GL}_{r_k}$ which are
 unramified outside of $S$, the central characters are of finite order,
 and the Hecke eigenvalues at $x\not\in S$ of $\pi$ is the disjoint
 union of that of $\pi^{i}$. By induction
 hypothesis, we have the overconvergent $F$-isocrystal $E_{\pi^i}$ of rank
 $r_i$. By construction, the Frobenius eigenvalue of $E$ and the
 semi-stable overconvergent $F$-isocrystal
 $E':=\bigoplus_{i=1}^kE_{\pi_i}$ are the same for any point outside of
 $S$. We know that $E'$ is $\iota$-pure of weight $0$. This shows that
 $E$ is $\iota$-pure of weight $0$ as well, and by applying Proposition
 \ref{p-adicCebo}, we have $E\cong E'$, which is a contradiction, and we
 conclude the proof of the claim.
\end{proof}

\begin{cl}
 \label{coinLE}
 Assume $\pi\in\mc{A}_r$ (resp.\ $\pi'\in\mc{A}_{r'}$) correspond in the
 sense of Langlands to $E\in\mc{I}_r$ (resp.\ $E'\in\mc{I}_{r'}$). Then
 we get
 \begin{equation*}
  L_x(\pi_x\times\pi'_x,t)=L_x(E_x\otimes E'_x,t),\qquad
   \varepsilon_x(\pi_x\times\pi'_x,t,\psi_x)=\varepsilon_x
   (E_x\otimes E'_x,t,\psi_x),
 \end{equation*}
 for any $x\in|X|$.
\end{cl}
\begin{proof}
 The proof is exactly the same as that of \cite[Prop VI.11
 (ii)]{La}. The assumption is slightly milder than {\it ibid.}\ since we
 already know that the generalized Ramanujan-Petersson conjecture
 [{\it ibid.}, VI.10 (i)] is true. In particular, $\pi_x$ is tempered for
 any $x\in|X|$, and $E$ and $E'$ are $\iota$-pure of weight $0$. In the
 proof, we need to replace [{\em ibid}., VI.5] by \cite{Ke} or
 \cite{AC}.
\end{proof}

Let us prove the theorem. Assume that the correspondence is established
for $r$ strictly less than $r_0$. Then Claim \ref{coinLE} is showing
that the local $L$-factors and $\varepsilon$-factors coincide via the
correspondence at any closed point of $X$. This enables us to apply
Claim \ref{indDeligh}, which give us the map $\pi_\bullet$ for
$r=r_0$.

It remains to construct $E_{\bullet}$. Since we are assuming Conjecture
(D), for $\pi\in\mc{A}_{r_0}$ with the set of unramified places $U$,
there exists an overconvergent $F$-isocrystal $E$ of rank $r_0$ whose
set of Frobenius eigenvalues at $x\in U$ is equal to that of Hecke
eigenvalues of $\pi$ at $x$. We need to prove that $E$ is
irreducible. Assume $E$ were not irreducible, and write
$E^{\mr{ss}}\cong\bigoplus E_i$ where $E_i$ are irreducible. By Lemma
\ref{twistdetfin}\footnote{
We do not have circular reasoning.},
there exists a twist $\chi_i$ such that $E_i(\chi_i)$ is of finite
determinant for any $i$. Take a
prime $\ell\neq p$, and let $\ms{F}_i$ (resp.\ $\ms{F}$) be
the irreducible $\ell$-adic sheaf corresponding to $\pi_{E_i(\chi_i)}$
(resp.\ $\pi$) in the sense of Langlands using the induction hypothesis
and proven Langlands correspondence for $\ell$-adic sheaves. Then the set of
Frobenius eigenvalues of $\ms{F}$ and $\bigoplus\ms{F}_i(\chi_i^{-1})$
(using the notation of Remark \ref{remtwist} (ii)) coincide at any
$x\in U$. By the \v{C}ebotarev density theorem, this is
not possible, which implies that $E$ is irreducible. We see from the
Frobenius eigenvalues at each closed point in $U$ that $E$ is of finite
determinant, and thus the theorem follows.

\section{Some consequences}
Finally, let us see some consequences of the theorem.

A {\em twist} is an $F$-isocrystal of rank $1$ on $\mr{Spec}(k)$. Let
$\chi$ be a twist. For an overconvergent $F$-isocrystal $E$, we denote
by $E(\chi)$ the tensor product $E\otimes f^*(\chi)$ where $f\colon
X\rightarrow\mr{Spec}(k)$ is the structural morphism.

\begin{lem}
 \label{twistdetfin}
 Let $X$ be a scheme of finite type over $k$.

 (i) Let $E$ be an overconvergent $F$-isocrystal on $X/K$ of rank
 $1$. Then there exists a twist $\chi$ and a positive integer $n$ such
 that $E(\chi)^{\otimes n}$ is trivial.

 (ii) For any overconvergent $F$-isocrystal $E$ on $X/K$, by taking an
 extension of $K$ if necessary, there exists a twist $\chi$ such that
 $E(\chi)$ is of finite determinant.
\end{lem}
\begin{proof}
 Let us see (i). We may replace $k$ by its finite extension. Thus
 we may assume that the residue field of $K$ is $k$ and there is a
 uniformizer of $K$ fixed by the Frobenius automorphism of $K$. By the
 definition of overconvergent $F$-isocrystals, we may
 assume that $X$ is reduced. By \cite[2.1.11]{Be}, we may shrink $X$,
 and in particular, we may assume that $X$ is smooth.
 Since $E$ is of rank $1$, there exists a twist $\chi'$ such that
 $E(\chi')$ is unit-root. Assume there exists a smooth compactification
 $X\hookrightarrow\overline{X}$ such that the complement is a simple
 normal crossing divisor. Then by the same argument as
 \cite[4.13]{Crrep}, using \cite[Thm 4.3]{Sh} and \cite[Theorem 2]{KL}
 instead of \cite[4.12]{Crrep} and the class field theory, we get that
 (i) holds in this case.

 There exists a generically \'{e}tale alteration $f\colon Y\rightarrow
 X$ such that $Y$ possesses a smooth compactification whose complement
 is a simple normal crossing divisor. This shows that there exists an
 integer $n'$ and a twist $\chi$ such that $f^*E(\chi)^{\otimes n'}$ is
 trivial. There exists open dense subschemes $V$ of $Y$ and $U$ of $X$
 such that $f_V\colon V\rightarrow U$ is finite \'{e}tale of degree
 $d$. Let $G$ be an $F$-isocrystal $U$ such that $f_V^*G$ is
 trivial. Then $G^{\otimes d}$ is trivial. This implies that
 $E(\chi)^{\otimes dn'}|U$ is trivial, and by using \cite[2.1.11]{Be}
 again, we conclude that $E(\chi)^{\otimes dn'}$ is trivial on $X$.

 For (ii), we only need to note that
 $\det(E(\chi))\cong\det(E)(\chi^{\otimes\mr{rk}(E)})$.
\end{proof}

\begin{rem}
 \label{remtwist}
 (i) The lemma also holds when $k$ is the perfection of an absolutely
 finitely generated field, since the theorem of \cite{KL} is still
 applicable in this case.

 (ii) Fixing a twist is equivalent to fixing an element of
 $K^\times$. Fix an isomorphism $\iota'\colon\overline{\mb{Q}}_p\cong
 \overline{\mb{Q}}_{\ell}$, and take a twist $\chi$ corresponding to
 $b\in\overline{\mb{Q}}_p$. Given a $\overline{\mb{Q}}_\ell$-sheaf
 $\ms{F}$, we have $\ms{F}^{\iota'(b)}$ using the notation of
 \cite[1.2.7]{WeilII}. We sometimes denote this sheaf by
 $\ms{F}(\chi)$.
\end{rem}

We say that a scheme $X$ over $k$ is a {\em d-variety} if there exists a
proper smooth formal scheme $\ms{P}$, a divisor $Z$ of the special
fiber of $\ms{P}$, and an embedding $X\hookrightarrow\ms{P}$ such that
$X=\overline{X}\setminus Z$ where $\overline{X}$ denotes the closure of
$X$ in $\ms{P}$. We are able to consider overholonomic
$F$-$\DdagQ{X}$-complexes as in \cite{AC}. We say that an overholonomic
$F$-$\DdagQ{X}$-complex $\ms{C}$ is {\em $\iota$-mixed} if there exists
a finite stratification $\{X_i\}$ by smooth d-varieties such that
$\ms{H}^j\mb{R}\underline{\Gamma}^\dag_{X_i}(\ms{C})$ is the
specialization of an $\iota$-mixed overconvergent $F$-isocrystal for any
$i$ and $j$. See \cite{AC} for more details.

\begin{cor}
 \label{ocisocimixed}
 Assume Conjecture {\normalfont (D)} holds for any function field. Then
 for any $d$-variety $X$ of finite type over $k$, any
 overholonomic $F$-$\DdagQ{X}$-complex is $\iota$-mixed.
\end{cor}
\begin{proof}
 By definition, it suffices to show that, for any {\em smooth} d-variety
 $X$, any irreducible overconvergent $F$-isocrystal $E$ on $X$ is
 $\iota$-pure. By Lemma \ref{twistdetfin} (ii), we may assume that $E$
 is of finite determinant. For any closed point $x$ of $X$, there exists
 a smooth curve $i\colon C\hookrightarrow X$ passing through $x$. It
 suffices to see that $i^*E$ is $\iota$-pure of weight $0$. Using Remark
 \ref{remRPconj} and Conjecture (L), any overconvergent $F$-isocrystal
 with finite determinant on a smooth curve is $\iota$-pure of weight
 $0$, which shows that $i^*E$ is $\iota$-pure of weight $0$, and
 conclude the proof.
\end{proof}

\begin{lem}
 The Langlands correspondence is established for $r=1$.
\end{lem}
\begin{proof}
 Let us construct $E_{\bullet}$. Using the reciprocity map, we have a
 representation of $\pi_1(U)$ with finite monodromy at the boundary,
 which induces a unit-root overconvergent $F$-isocrystal by
 \cite[7.1.1]{Ts}. To construct $\pi_{\bullet}$, we note that objects in
 $\mc{I}_1$ are unit-root since they are finite. Thus using the result
 of Tsuzuki and the reciprocity map, we conclude.
\end{proof}

\begin{thm}
 Any overconvergent $F$-isocrystal of rank less than or equal to $2$ on
 a smooth variety is $\iota$-mixed.
\end{thm}
\begin{proof}
 Arguing in the same way as Corollary \ref{ocisocimixed}, we are reduced
 to showing the existence of $\pi_\bullet$ for $r=1,2$.
 The case $r=1$ has been established by the previous lemma. In this
 case, the coincidence of local $\varepsilon$-factors for any
 place of $X$ follows by definition. Thus applying Claim
 \ref{indDeligh}, we have the map $\pi_\bullet$ for $r=2$, and the
 theorem follows.
\end{proof}


Tomoyuki Abe:\\
Institute for the Physics and Mathematics of the Universe (IPMU)\\
The University of Tokyo\\
5-1-5 Kashiwanoha, Kashiwa, Chiba, 277-8583, Japan \\
e-mail: {\tt tomoyuki.abe@ipmu.jp}

\end{document}